\documentclass[11 pt]{amsart}

\usepackage{amssymb}
\usepackage{eucal}
\usepackage{amsrefs}

\theoremstyle{plain}
\newtheorem{theorem}{Theorem}

\newtheorem{lemma}[theorem]{Lemma}

\theoremstyle{remark}
\newtheorem*{notations}{Notations and conventions}

\theoremstyle{definition}

\numberwithin{equation}{subsection}

\allowdisplaybreaks

\DeclareMathOperator{\ad}{ad}

\DeclareMathOperator{\Gal}{Gal}

\DeclareMathOperator{\id}{id}
\DeclareMathOperator{\Lie}{Lie}
\DeclareMathOperator{\Res}{Res}

\DeclareMathOperator{\Tr}{Tr}

\newcommand{\bQ}{\mathbb Q}
\newcommand{\bR}{\mathbb R}
\newcommand{\bC}{\mathbb C}
\newcommand{\bH}{\mathbb H}

\newcommand{\sA}{\mathcal A}
\newcommand{\sG}{\mathcal G}

\newcommand{\sV}{\mathcal V}

\newcommand{\germg}{\mathfrak g}
\newcommand{\germk}{\mathfrak k}
\newcommand{\germl}{\mathfrak l}
\newcommand{\germp}{\mathfrak p}
\newcommand{\germs}{\mathfrak s}
\newcommand{\germu}{\mathfrak u}
\newcommand{\germS}{\mathfrak S}

\newcommand{\tE}{\widetilde E}
\newcommand{\tG}{\widetilde G}
\newcommand{\tJ}{\widetilde J}
\newcommand{\tV}{\widetilde V}

\begin{document}

\title{Classification of Kuga fiber varieties}

\author{Salman Abdulali}
\address{Department of Mathematics, Mail Stop 561, East Carolina University, Greenville, NC 27858, USA}
\email{abdulalis@ecu.edu}


\begin{abstract}
We  complete Satake's classification of Kuga fiber varieties by showing that if a representation $\rho$ of a hermitian algebraic group satisfies Satake's necessary conditions, then some multiple of $\rho$ defines a Kuga fiber variety.
\end{abstract}

\maketitle

\section{Introduction}

Kuga fiber varieties \cites{Kuga, Kuga2018} are families of abelian varieties $\sA \to \sV$, where $\sV = \Gamma \backslash G(\bR)^0 / K$ is an arithmetic variety, and $\sA$ is the pullback of the universal family of abelian varieties over a Siegel modular variety.
Here, $G$ is a semisimple algebraic group over $\bQ$ such that $G(\bR)$ is of hermitian type, $K$ is a maximal compact subgroup of $G(\bR)^0$, and $\Gamma$ is an arithmetic subgroup of $G(\bQ)$.
A Kuga fiber variety is constructed from a symplectic representation $\rho \colon G \to Sp(2n, \bQ)$ which is equivariant with a holomorphic map $\tau \colon X \to \germS_n$, where $X = G(\bR)^0 / K$ is the symmetric domain belonging to $G$, and $\germS_n$ is the Siegel space of degree $n$.
Kuga assumed that $\sV$ is compact; we do not make this assumption.

Kuga's original motivation was to prove the Ramanujan conjecture, a goal achieved by Deligne \cites{Deligne, DeligneWeil}.
Kuga fiber varieties, which include Shimura's \textsc{pel}-families \cite{Shimura}, have played a central role in the arithmetic theory of automorphic forms \cites{KugaShimura, Ohta1, Ohta2}. These varieties are also key to the study of algebraic cycles on abelian varieties and abelian schemes \cites{Abdulali1994c, Abdulali2002, Gordon1, Gordon2, HallKuga, Kuga1982, Mumford1969, Tjiok}; indeed, the concept of the Hodge group (or Mumford-Tate group) of an abelian variety arose in the context of Kuga fiber varieties \cite{Mumford1966}. Another area in which Kuga fiber varieties play a key role is in the study of $K3$-surfaces, via the Kuga-Satake construction of abelian varieties associated to $K3$-surfaces \cite{KugaSatake}, as in Deligne's proof of the Weil conjectures for these surfaces \cite{DeligneK3}.

We consider the following problem in this paper: Given an arithmetic variety $\sV$, classify all Kuga fiber varieties over it.
Equivalently, given the group $G$, find all representations of it into a symplectic group which are equivariant with holomorphic maps of the corresponding symmetric domains.
From another point of view, this problem is equivalent to the classification (up to isogeny) of the semisimple parts of the Hodge groups of abelian varieties, together with their action on the first cohomology of the abelian variety.
This problem was raised by Kuga \cite{Kuga} in the 1960's, and partially answered by Satake \cites{Satake1964, Satake1965a, Satake1965b, SatakeBoulder, Satake1966, Satake1967, Satakebook, Satake1997}.
Addington \cite{Addington} completed Satake's classification for $\bQ$-simple groups of type II (orthogonal groups) and type III (symplectic groups).
This was partially extended to non-simple groups of type III by Abdulali \cites{Abdulali1, Abdulali1988}.

Deligne \cite{Deligne1979} and Milne \cite{Milne2013} considered this problem from a somewhat different point of view. Their results are similar to those of Satake.
Milne states the theorem for all suitable representations of a $\bQ$-simple group, though he proves it only in the situations where Satake proved it, and he does not deal fully with non-simple groups.
It is important to deal with non-simple groups because the semisimple part of the Hodge group of a simple abelian variety need not be simple.
We give an example of such an abelian variety in Section \ref{example}.
Further examples may be found in Satake \cite{SatakeBoulder}*{Remark 2, p.~356}, Kuga \cite{Kuga1984}*{\S 5}, and Abdulali \citelist{\cite{Abdulali1988}*{\S 4} \cite{Abdulali1994a}*{\S 2.4}}.

Green, Griffiths, and Kerr \cites{GreenGriffithsKerr2012, GreenGriffithsKerr2013} and Patrikis \cite{Patrikis} have considered the more general problem of classifying the Hodge groups of Hodge structures of higher weight.
They completely classify the reductive groups which are Mumford-Tate groups of polarizable Hodge structures of arbitrary weight; however the
representations of the groups on the Hodge structures have not been classified.

The key to our classification is a reduction to the \emph{rigid} case, which is much easier.
In the proof of our main theorem (Theorem~\ref{maintheorem}) we reduce the general case to the rigid case,
which is proved in Theorem~\ref{rigid}. The inspiration for this strategy comes from the construction of families of families of abelian varieties by Kuga and Ihara \cite{KugaIhara}, and the related concept of ``sharing'' in Kuga \cite{Kuga1984}*{\S 5, p.~277}.

\begin{notations}
All representations are finite-dimensional and algebraic.
For a finite field extension $E$ of a field $F$, we let $\Res_{E/F}$ be
the restriction of scalars functor, from schemes over $E$ to schemes over~$F$.
For an algebraic or topological group $G$, we denote by $G^0$ the connected component of the identity.
\end{notations}

\section{Kuga fiber varieties}
\label{kuga}
In this section we give an overview of the construction of Kuga fiber varieties. Our primary purpose is to fix the notations and terminology; for details we refer to Satake \cite{Satakebook}.

\subsection{Groups of hermitian type}
Let $G$ be a group of hermitian type.
This means that $G$ is a semisimple real Lie group, and $X := G^0/K$ is a bounded symmetric domain for a maximal compact subgroup $K$ of $G^0$.
Denote by $1$ the identity element of $G$.
Let $\germg := \Lie G$ be the Lie algebra of $G$, $\germk := \Lie K$, and let
$\germg = \germk \oplus \germp$ be the corresponding Cartan decomposition.
Differentiating the natural map $\nu \colon G^0 \to X$ induces an isomorphism of
$\germp$ with $T_o(X)$, the tangent space of $X$ at the base point $o = \nu(1)$, and there exists
a unique $H_0 \in Z(\germk)$, called the \emph{$H$-element} at $o$,
such that $\ad H_0 | \germp$ is the complex structure on $T_o (X)$.

\subsection{Equivariant holomorphic maps}
Let $G_1$ and $G_2$ be groups of hermitian type, with symmetric spaces $X_1$ and $X_2$ respectively.
Let $H_0$ and $H'_0$ be $H$-elements at base points $o_1 \in X_1$ and $o_2 \in X_2$ respectively.
Let $\rho \colon G_1 \to G_2$ be a homomorphism of Lie groups.
We say that $\rho$ satisfies the \emph{$H_1$-condition} relative to the $H$-elements
$H_0$ and $H'_0$ if
\begin{equation}
\label{h1}
	[d \rho (H_0) - H'_0, d \rho (g)] = 0 \qquad \text{for all }g \in \germg.
\end{equation}
The stronger condition
\begin{equation}
\label{h2}
	d \rho (H_0) = H'_0
\end{equation}
is called the \emph{$H_2$-condition}. 
If either of these is satisfied, then there exists a unique holomorphic map
$\tau \colon X_1 \to X_2$ such that $\tau (o_1) = o_2$, and the pair $(\rho, \tau)$
is equivariant in the sense that
\[
	\tau (g \cdot x) = \rho(g) \cdot \tau (x) \qquad \text{for all } g \in G^0, \ x \in X.
\]
In fact, Clozel \cite{Clozel} has shown that if $G_2$ has no exceptional factors, then the $H_1$-condition is equivalent to the existence of an equivariant holomorphic map.

\subsection{The Siegel space}
Let $E$ be a nondegenerate alternating form on a finite-dimensional real vector space $V$.
The symplectic group $Sp(V, E)$ is a Lie group of hermitian type;
the associated symmetric domain is the \emph{Siegel space}
\begin{multline*}
	\germS(V, E) = \{\, J \in GL(V) \mid J^2 = -I \text{ and } \\
	E(x,Jy) \text{ is symmetric, positive definite}\, \}.
\end{multline*}
$Sp(V, E)$ acts on $\germS(V, E)$ by conjugation.
The $H$-element at $J \in \germS(V, E)$ is~$J/2$.

\begin{lemma}
\label{h2lemma}
Let $G$ be a group of hermitian type with symmetric domain $X$, and let $E$ be a nondegenerate alternating form on a finite-dimensional real vector space $V$.
Let $\rho \colon G \to Sp(V, E)$ satisfy the $H_2$-condition with respect to $H$-elements $H_0$ and $H'_0 = J_0/2$ at base points $o \in X$ and $J_0 \in \germS(V, E)$, respectively.
Then $J_0 \in \rho(G)$.
\end{lemma}

\begin{proof}
Since $J_0$ is a complex structure on $V$, there exists a basis of $V$ with respect to which
$J_0 = \begin{pmatrix} 0 & I_n \\ -I_n & 0 \end{pmatrix}$,
where  $2n = \dim V$.
Then, using the $H_2$-condition, we calculate that
$J_0 = \exp (\frac{\pi}{2} J_0) = \exp(\pi H'_0) = \exp(d\rho (\pi H_0)) = \rho (\exp(\pi H_0))$.
\end{proof}

\subsection{The fiber varieties}
We shall say that an algebraic group $G$ over a subfield of $\bR$ is of hermitian type if the Lie group $G(\bR)$ is of hermitian type.
Now let $G$ be a connected, semisimple algebraic group of hermitian type over $\bQ$. Assume that $G$ has no nontrivial, connected, normal subgroup $H$ such that $H(\bR)$ is compact.

Let $E$ be a nondegenerate alternating form on a finite-dimensional rational vector space $V$.
The symplectic group $Sp(V, E)$ is then a $\bQ$-algebraic group of hermitian type.
We write $\germS(V, E)$ for $\germS(V_{\bR}, E_{\bR})$.
Let $\rho \colon G \to Sp(V, E)$ be a representation defined over $\bQ$, which satisfies the $H_1$ condition with respect to the $H$-elements $H_0$ and $H'_0 = J/2$.
Let $\tau \colon X \to \germS(V, E)$ be the corresponding equivariant holomorphic map.
Let $\Gamma$ be a torsion-free arithmetic subgroup of $G(\bQ)$,
and $L$ a lattice in $V$ such that $\rho(\Gamma)L = L$.
Then the natural map
\[
	\sA = (\Gamma \ltimes_\rho L) \backslash (X \times V_{\bR})
	\longrightarrow \sV :=  \Gamma \backslash X
\]
is a morphism of smooth quasiprojective algebraic varieties
(Borel \cite{Borel1972}*{Theorem 3.10, p.~559} and Deligne \cite{900}*{p.~74}),
so that $\sA$ is a fiber variety over~$\sV$ called a \emph{Kuga fiber variety}.
The fiber $\sA_P$ over any point $P \in \sV$ is an abelian variety isomorphic to the torus
$V_{\bR}/L$ with the complex structure $\tau (x)$, where $x$ is a point in $X$ lying over $P$.

We say that a representation $\rho \colon G \to GL(V)$ defines a Kuga fiber variety if $\rho (G)$ is contained in a symplectic group $Sp(V, E)$, and $\rho$ satisfies the $H_1$-condition with respect to some $H$-elements.

\section{Satake's Classification}

\subsection{Necessary Conditions}
\label{necessary}
In a series of papers \cites{Satake1964, Satake1965a, Satake1965b, SatakeBoulder, Satake1966, Satake1967, Satakebook, Satake1997} Satake classified the $H_1$-representations of a given hermitian group into a symplectic group.
We summarize his results below.
Let $G$ be a connected, semisimple, linear algebraic group over $\bQ$.
Assume that $G(\bR)^0$ is of hermitian type, and has no nontrivial, connected, normal $\bQ$-subgroup $H$ with $H(\bR)$ compact.
After replacing $G$ by a finite covering, if necessary, we may write
\[
	\germg_{\bR} = \bigoplus_{j=0}^s \germg_j, \qquad G(\bR) = G_0 \times G_1 \times \dots \times G_s,
\]
where $G_0$ is compact, each $G_j$ is a noncompact absolutely simple Lie group for $j > 0$, and, each $\germg_j = \Lie(G_j)$.
Suppose $\rho \colon G \to Sp(V, E)$ is a symplectic representation satisfying the $H_1$-condition.
Then,

\begin{enumerate}
\item For $j = 1, \dots s$, we have that $G_j$ is one of the following:
\begin{enumerate}
	\item Type I:
		$SU(p, q)$ with $p \geq q \geq 1$;
	\item Type II:
		$SU^-(n, \bH)$ with $n \geq 5$ (this is the group that Helgason \cite{Helgason}*{p.~445} calls $SO^{\star}(2n)$);
	\item Type III:
		$Sp(2n, \bR)$ with $n \geq 1$;
	\item Type IV:
		$Spin(p, 2)$ with $p \geq 1, p \neq 2$.
\end{enumerate}
\item \label{stability} Let $\rho'$ be a nontrivial $\bC$-irreducible subrepresentation of $\rho_{\bC}$.
Then, for some index $j$ ($1 \leq j \leq s$), we have that  $\rho'$ is equivalent to $\rho_0 \otimes \rho_j$, where $\rho_0$ is a representation of $G_{0,\bC}$, and, $\rho_j$ is a representation of $G_{j,\bC}$.
We call this the \emph{stability} condition.
\item Fix an index $j$ with $1 \leq j \leq s$. Let $\rho_j$ be an irreducible subrepresentation of $V_{\bR}$ considered as a $G_j$-module. Then $\rho_j$ is either trivial or given by one of the following:
\begin{enumerate}
	\item
	\label{supq}
	If $G_j = SU(p,q)$ with $p \geq q \geq 2$, then $\rho_{j,\bC}$ is the direct sum of the standard representation of $G_{j,\bC} = SL_{p+q}(\bC)$ and its contragredient; it satisfies the $H_2$-condition if and only if $p=q$.
	\item If $G_j = SU(p, 1)$, then $\rho_j$ is one of the following:
		\begin{enumerate}
			\item\label{sup1gen} $\bigwedge^k \oplus \bigwedge^{p+1-k}$, for some $k$ with $1 \leq k < \frac{p+1}{2}$;
			\item $\bigwedge^k$ with $k = \frac{p+1}{2}$, and $p \equiv 1 \pmod{4}$;
			\item the direct sum of two copies of $\bigwedge^k$ with $k = \frac{p+1}{2}$, and $p \equiv 3 \pmod{4}$.
		 \end{enumerate}
	The $H_2$-condition is satisfied if and only if $k = \frac{p+1}{2}$.
	\item If $G_j = SU^-(n, \bH)$ with $n \geq 5$, then $\rho_{j,\bC}$ is the direct sum of two copies of the standard representation. The $H_2$-condition is satisfied in this case.
	\item If $G_j = Sp(2n, \bR)$, then $\rho_j$ is the standard representation, and satisfies the $H_2$-condition.
	\item If $G_j = Spin(p,2)$ with $p \geq 1$ and $p$ odd, then
		\begin{enumerate}
			\item $\rho_j$ is the spin representation if $p \equiv 1, 3 \pmod{8}$;
			\item $\rho_{j,\bC}$ is the direct sum of two copies of the spin representation if $p \equiv 5, 7 \pmod{8}$.
		\end{enumerate}
	In both cases, $\rho_j$ satisfies the $H_2$-condition.
	\item If $G_j = Spin(p,2)$ with $p \geq 4$, and $p$ even, then $\rho_j$ is
	\begin{enumerate}
		\item one of the two spin representations if $p \equiv 2 \pmod{8}$;
		\item the direct sum of two copies of a spin representation if $p \equiv 6 \pmod{8}$;
		\item the direct sum of the two spin representations if $p \equiv 0 \pmod{4}$.
	\end{enumerate}
	In each case, $\rho_j$ satisfies the $H_2$-condition.
\end{enumerate}
\end{enumerate}

We note that the above conditions imply that $\rho$ is self-dual.

\subsection{A sufficient condition}
Satake showed that the necessary conditions listed in \S \ref{necessary} are sufficient if we make an additional assumption.

\begin{theorem}[Satake \cite{Satake1967}]
\label{sataketype}
Let $G$ be a $\bQ$-simple hermitian group, and write $G(\bR) = \prod_{\alpha \in S} G_{\alpha}$ where each $G_{\alpha}$ is an absolutely simple real algebraic group.
Let $\rho$ be a representation of $G$ satisfying the conditions of \S\textup{\ref{necessary}}.
Assume further that each irreducible subrepresentation of $\rho_{\bC}$ is nontrivial on $G_{\alpha}$ for exactly one $\alpha$.
Then some multiple of $\rho$ defines a Kuga fiber variety.
\end{theorem}

\subsection{More on type I}
\label{nonh2}
We now take a closer look at the $H_1$-representation $\rho \colon SU(p,q) \to Sp(V, E)$ given by item (\ref{supq}) of the list in \S\ref{necessary}.
We recall the matrix representation of this given by Satake.
Let $J_0 \in \germS(V, E)$ be the base point.
The eigenvalues of $J_0$ on $V_{\bC}$ are $i$ and $-i$, and we take a basis of $V_{\bC}$ with respect to which the matrix of $J_0$ is
$\begin{pmatrix} iI_m & 0 \\ 0 & -iI_m \end{pmatrix}$.
The Lie algebra of $Sp(V, E)$ with respect to this basis is given by
\[
\germs\germp (V, E) =
\left\{ \begin{pmatrix} A & B \\ C & -{}^tA \end{pmatrix} \, \middle| \, B, C \text{ symmetric} \right\},
\]
and the $H$-element is $H'_0 = \begin{pmatrix} \frac{i}{2}I_m & 0 \\ 0 & -\frac{i}{2}I_m \end{pmatrix}$.

With respect to a suitable basis the Lie algebra of $SU(p,q)$ is given by
\[
	\germs\germu(p,q) = \left\{
	\begin{pmatrix} X_1 & X_{12} \\ {}^t \overline{X}_{12} & X_2 \end{pmatrix} \in \germs\germl_{p+q}(\bC) \, \middle| \,
	\begin{matrix} X_1 \in M_p(\bC), X_2 \in M_q(\bC) \\ {}^t\overline{X}_j = -X_j (j=1,2) \end{matrix}
	\right\},
\]
and an $H$-element is given by
\[
	H_0 = \begin{pmatrix} \frac{qi}{p+q} I_p & 0 \\ 0 & -\frac{pi}{p+q} I_q \end{pmatrix}.
\]
Then, $d\rho \colon \germs\germu(p,q) \to \germs\germp_{2p+2q}$ is given by
\[
	 \begin{pmatrix} X_1 & X_{12} \\ {}^t \overline{X}_{12} & X_2 \end{pmatrix} \mapsto
	 \begin{pmatrix} \overline{X}_2 & 0 & 0 & {}^tX_{12} \\
	 0 & X_1 & X_{12} & 0 \\
	 0 & {}^t\overline{X}_{12} & X_2 &0 \\
	 \overline{X}_{12} & 0 & 0 & \overline{X}_1 \end{pmatrix}.
\]
We extend $d\rho$ to $\germu(p,q)$, and denote by
\[
	\bar{\rho} \colon U(p,q) \to Sp(2p+2q, \bR)
\]
the corresponding map of Lie groups which extends $\rho$.
Let
\[
	\bar{H}_0^{p,q} = \begin{pmatrix} \frac{i}{2}I_p & 0 \\ 0 & -\frac{i}{2}I_q \end{pmatrix}.
\]
Then $d\bar{\rho} (\bar{H}_0) = H'_0$.
It follows, as in the proof of Lemma \ref{h2lemma}, that $J_0 \in \bar{\rho} (U(p,q))$. 

Consider, next, the $H_1$-representation $\rho \colon SU(p, 1) \to Sp(2m, \bR)$ given in item (\ref{sup1gen}) of the list in \S \ref{necessary}.
We extend it to a representation $\bar{\rho} \colon U(p, 1) \to Sp(2m, \bR)$.
Let
\[
	\tilde{H}_0 = \begin{pmatrix} \frac{i}{2k}I_p & 0 \\ 0 & \frac{1-2k}{2k}i \end{pmatrix}.
\]
Then
\[
	\bigwedge^k (\tilde{H}_0) = \begin{pmatrix} \frac{i}{2}I_p' & 0 \\ 0 & -\frac{i}{2}I_q' \end{pmatrix}
	= \bar{H}_0^{p',q'},
\]
where $p' = \binom{p}{k}$ and $q' = \binom{p}{k-1}$.
From this we see that $d\bar{\rho} (\tilde{H}_0) = H'_0$, the $H$-element of $Sp(2m, \bR)$.
It follows, as in the proof of Lemma \ref{h2lemma}, that $J_0 \in \bar{\rho} (U(p,1))$. 

\section{The rigid case}

\subsection{Statement of the theorem}

Let $G$ be an algebraic group over $\bQ$ of hermitian type, and $\rho$ a representation of $G$ satisfying Satake's conditions in \S\ref{necessary}.
By the stability condition (\ref{stability}) of \S\ref{necessary}, every irreducible subrepresentation of $\rho_{\bC}$ is nontrivial on at most one noncompact factor of $G(\bR)$.
We say that $\rho$ is \emph{rigid}, if every irreducible subrepresentation of $\rho_{\bC}$ is nontrivial on exactly one noncompact factor of $G(\bR)$.
We begin by classifying the rigid representations which define Kuga fiber varieties.

\begin{theorem}
\label{rigid}
Let $G$ be a semisimple connected algebraic group over $\bQ$ such that $G(\bR)^0$ is of hermitian type and has no compact factors defined over $\bQ$.
Let $\rho$ be a representation of $G$ satisfying Satake's conditions in \S\textup{\ref{necessary}}.
If $\rho$ is rigid, then some multiple of $\rho$ defines a Kuga fiber variety.
\end{theorem}

The rest of this section is devoted to the proof of this theorem.

\subsection{Beginning of the proof}
Without loss of generality we assume that $G$ is simply connected, and $\rho$ is nontrivial and a multiple of a $\bQ$-irreducible representation (see Satake \cite{Satakebook}*{p.~189}).

Write $G = \prod_{j=1}^t G_j$, where each $G_j$ is a simple group of hermitian type.
Then there are totally real number fields $F_j$, and absolutely simple groups $\tG_j$ over $F_j$, such that $G_j = \Res_{F_j/\bQ} \tG_j$ for $1 \leq j \leq t$.
Let $F$ be the smallest Galois extension of $\bQ$ containing all the $F_j$, and $\sG = \Gal (F/\bQ)$.
Let $S_j$ be the set of embeddings of $F_j$ into $\bR$, and let $S$ be the disjoint union of the $S_j$'s.
For $\alpha \in S$, we let $j(\alpha)$ be the unique index such that $\alpha \in S_{j(\alpha)}$.
We note that $F$ is a totally real field, $\sG$ acts on $S$, and the orbits of this action are the sets $S_j$.
For $\alpha \in S$, we let $G_{\alpha} = \widetilde{G}_{j(\alpha)} \otimes_{F_{j(\alpha)}, \alpha} F$.
Then $G_F = \prod_{\alpha \in S} G_{\alpha}$.

Let $S_0 = \{ \alpha \in S \mid G_{\alpha, \bR} \text{ is not compact} \}$.
An $H$-element of $G_{\bR}$ is given by $H_0 = \sum_{\alpha \in S_0} H_{0, \alpha}$,
where $H_{0, \alpha}$ is an $H$-element of $G_{\alpha, \bR}$.

Let $\rho_0 \colon G_{\bC} \to GL(V_0)$ be a $\bC$-irreducible subrepresentation of $\rho_{\bC}$.
Let
\[
	M = \{ \alpha \in S \mid \rho_0 \text{ is nontrivial on } G_{\alpha,\bC} \},
\]
and let $\alpha_0$ be the unique element of $M$ such that $G_{\alpha_0, \bR}$ is not compact.
Write $\rho_0 = \otimes_{\alpha \in M} \rho'_{\alpha}$, where $\rho'_{\alpha}$ is an irreducible representation of $G_{\alpha, \bC}$.
Either $\rho'_{\alpha}$, or the sum of two copies of $\rho'_{\alpha}$, or, the direct sum of $\rho'_{\alpha}$ and its complex conjugate is defined over $\bR$.
Let $\rho_{\alpha}$ be the real representation such that $\rho'_{\alpha}$ equals $\rho_{\alpha, \bC}$, the direct sum of two copies of $\rho_{\alpha, \bC}$, or the direct sum of $\rho_{\alpha, \bC}$ and its complex conjugate, respectively.

For each $\alpha \in M$, let $\hat{\rho}_{\alpha} = \sum_{\sigma \in \sG} \rho_{\alpha}^{\sigma}$.
Then $\hat{\rho}_{\alpha}$ is a representation of $G_{j(\alpha)}$ satisfying the hypotheses of Theorem~\ref{sataketype},
so some multiple of it defines a Kuga fiber variety.
By Satake's construction (see \cite{Satakebook}*{\S IV.6, Theorems 6.1, 6.2, 6.3}), this representation is defined over $F_{j(\alpha)}$ in the sense that it is the restriction from $F_{j(\alpha)}$ to $\bQ$ of a symplectic representation
\[
	\tilde{\rho}_{\alpha} \colon \tG_{j(\alpha)} \to Sp(\tV_{\alpha}, \tE_{\alpha}).
\]
Here, $\tE_{\alpha}$ is an $F_{j(\alpha)}$-bilinear alternating form on $\tV_{\alpha}$, and $E_{\alpha} = \Tr_{{F_{j(\alpha)}}/\bQ} \tE_{\alpha}$ is a $G_{j(\alpha)}$-invariant $\bQ$-bilinear alternating form on $V_{\alpha} = \Res_{{F_{j(\alpha)}}/\bQ}\tV_{\alpha}$.
Let $V_{\alpha} = \Res_{F_{j(\alpha)}/{\bQ}} \tV_{\alpha}$.
Then $V_{\alpha} \otimes F_{j(\alpha)} = \oplus_{\sigma \in \sG} \tV_{\alpha}^{\sigma}$ (up to multiplicity), and $\tV_{\alpha}^{\sigma}$ is the representation space of $\rho_{\alpha}^{\sigma}$.
For each $\alpha \in S_{j(\alpha)}$, there is a complex structure $J_{\alpha}$ on $\tV_{\alpha,\bR}$ such that $\tE_{\alpha} (x, J_{\alpha}y)$ is a symmetric, positive definite form.
Moreover, $\rho_{\alpha_0} \colon G_{\alpha_0, \bR} \to Sp(\tV_{\alpha_0,\bR}, \tE_{\alpha_0})$ satisfies the $H_1$-condition with respect to the $H$-elements $H_{0, \alpha_0}$ and $\frac{1}{2} J_{\alpha_0}$.

\subsection{Construction of a symmetric form}
We claim that for each $\alpha \in M$ there exists a $G_{\alpha}$-invariant, $F_{j(\alpha)}$-bilinear positive definite symmetric form $\gamma_{\alpha}$ on $\tV_{\alpha}$.
The space of symmetric positive definite $G_{\alpha, \bR}$-invariant forms on $\tV_{\alpha,\bR}$ is an open subset of the space of symmetric forms; it is nonempty because $\tE_{\alpha} (x, J_{\alpha}y)$ is such a form. Therefore it contains an $F_{j(\alpha)}$-rational point $\gamma_{\alpha}$. This proves the claim.

We next claim that when $\alpha = \alpha_0$ we have
\begin{equation}
\label{J_invariance}
	\gamma_{\alpha_0} (J_{\alpha_0}x, J_{\alpha_0}y) = \gamma_{\alpha_0} (x,y).
\end{equation}
If $\rho_{\alpha_0}$ satisfies the $H_2$-condition, then Lemma \ref{h2lemma} shows that
$J_{\alpha_0} $ belongs to the image of $G_{\alpha_0}(\bR)$ under $\rho_{\alpha_0}$, so (\ref{J_invariance}) is a consequence of $\gamma_{\alpha_0}$ being $G_{\alpha_0}$-invariant.
Finally, we consider the situation when the $H_2$-condition is not satisfied.
Then we are in either case (\ref{supq}) or case (\ref{sup1gen}) of \S\ref{necessary}.
In both cases, $G_{\alpha}$ is a special unitary group $SU(W, h)$.
We can extend $\hat{\rho}_{\alpha}$ to a representation of the full unitary group $U(W, h)$.
We have seen in \S\ref{nonh2} that  $J_{\alpha_0}$ belongs to the image of $U(W,h)$, so now we can argue as before to prove (\ref{J_invariance}) in this situation.

Observe that (\ref{J_invariance}) is equivalent to
\begin{equation}
\label{J_alternate}
	\gamma_{\alpha_0} (x, J_{\alpha_0}y) = - \gamma_{\alpha_0} (y,J_{\alpha_0}x),
\end{equation}
since $J_{\alpha_0}$ is a complex structure.

\subsection{Construction of an alternating form}
Let $\tV = \otimes_{\alpha \in M} (\tV_{\alpha} \otimes_{F_{j(\alpha)}} F)$.
Define an $F$-bilinear alternating form $\tE$ on $\tV$ by
\[
	\tE (\otimes_{\alpha} x_{\alpha}, \otimes_{\alpha} y_{\alpha}) =
	\sum_{\alpha \in M} \Bigg( \tE_{\alpha} (x_{\alpha}, y_{\alpha}) \prod_{\substack{\beta \in M \\ \beta \neq \alpha}} \gamma_{\beta} (x_{\beta}, y_{\beta}) \Bigg).
\]
Then $\tE$ is $\tG$-invariant, where $\tG = \prod_{j=1}^t \tG_j$.
Next, we define a $\bQ$-bilinear alternating form $E$ on $V = \Res_{F/\bQ} \tV$ by
\[
	E(x,y) = \Tr_{F/\bQ} \tE(x,y).
\]
Then $E$ is $G$-invariant.

\subsection{Construction of a complex structure}
We next construct a complex structure $\tJ$ on $\tV_{\bR} = \otimes_{\alpha \in M} (\tV_{\alpha} \otimes_{F_{j(\alpha)}} \bR)$.
We have seen that we have a complex structure $J_{\alpha_0}$ on $\tV_{\alpha_0, \bR}$ such that
$\tE_{\alpha_0} (x, J_{\alpha_0} y)$ is symmetric and positive definite.
Define $\tJ$ on $\tV_{\bR}$ by
$\tJ(\otimes x_{\alpha}) = \otimes I_{\alpha} x_{\alpha}$,
where $I_{\alpha_0} = J_{\alpha_0}$, and $I_{\alpha}$ is the identity for $\alpha \neq \alpha_0$.
Then $\tJ$ is a complex structure.

For $x = \otimes x_{\alpha}, y = \otimes y_{\alpha} \in \tV_{\bR}$, we have
\begin{align*}
	\tE (x, \tJ y) & = \sum_{\alpha \in M} \Bigg( \tE_{\alpha} (x_{\alpha}, I_{\alpha} y_{\alpha}) \prod_{\substack{\beta \in M \\ \beta \neq \alpha}} \gamma_{\beta} (x_{\beta}, I_{\beta} y_{\beta}) \Bigg) \\
	& = \tE_{\alpha_0} (x_{\alpha_0}, J_{\alpha_0} y_{\alpha_0}) \prod_{\substack{\beta \in M \\ \beta \neq \alpha_0}} \gamma_{\beta} (x_{\beta}, y_{\beta}) \\
	& \qquad + \sum_{\substack{\alpha \in M \\ \alpha \neq \alpha_0}} \Bigg( \tE_{\alpha} (x_{\alpha}, y_{\alpha}) \gamma_{\alpha_0}(x_{\alpha_0}, J_{\alpha_0} y_{\alpha_0}) \prod_{\substack{\beta \in M \\ \beta \neq \alpha \\ \beta \neq \alpha_0}} \gamma_{\beta} (x_{\beta}, y_{\beta}) \Bigg)\\
	& = \tE_{\alpha_0} (y_{\alpha_0}, J_{\alpha_0} x_{\alpha_0}) \prod_{\substack{\beta \in M \\ \beta \neq \alpha_0}} \gamma_{\beta} (y_{\beta}, x_{\beta}) \\
	& \qquad + \sum_{\substack{\alpha \in M \\ \alpha \neq \alpha_0}} \Bigg( \tE_{\alpha} (y_{\alpha}, x_{\alpha}) \gamma_{\alpha_0}(y_{\alpha_0}, J_{\alpha_0} x_{\alpha_0}) \prod_{\substack{\beta \in M \\ \beta \neq \alpha \\ \beta \neq \alpha_0}} \gamma_{\beta} (y_{\beta}, x_{\beta}) \Bigg)\\
	& = \sum_{\alpha \in M} \Bigg( \tE_{\alpha} (y_{\alpha}, I_{\alpha} x_{\alpha}) \prod_{\substack{\beta \in M \\ \beta \neq \alpha}} \gamma_{\beta} (y_{\beta}, I_{\beta} x_{\beta}) \Bigg) \\
	& = \tE (y, \tJ x),
\end{align*}
because $\tE_{\alpha_0} (x, J_{\alpha_0} y)$ and $\gamma_{\alpha}$ are symmetric, $I_{\alpha}$ is the identity for $\alpha \neq \alpha_0$, $\tE_{\alpha_0} (x_{\alpha_0}, J_{\alpha_0} y_{\alpha_0})$ and $\gamma_{\alpha_0}(x_{\alpha_0}, J_{\alpha_0} y_{\alpha_0})$ are alternating, and using (\ref{J_alternate}).
Thus $\tE (x, \tJ y)$ is a symmetric form on $\tV_{\bR}$.
It follows that $E (x, \tJ y)$ is symmetric.

If $x=y$ we have
\begin{align*}
	\tE (x, \tJ x) &=
	 \sum_{\alpha \in M} \Bigg( \tE_{\alpha} (x_{\alpha}, I_{\alpha} x_{\alpha}) \prod_{\substack{\beta \in M \\ \beta \neq \alpha}} \gamma_{\beta} (x_{\beta}, I_{\beta} x_{\beta}) \Bigg) \\
	 &= \tE_{\alpha_0} (x_{\alpha_0}, J_{\alpha_0} x_{\alpha_0}) \prod_{\substack{\beta \in M \\ \beta \neq \alpha_0}} \gamma_{\beta} (x_{\beta}, x_{\beta}),
\end{align*}
so $\tE (x, \tJ y)$ is positive definite.

Our next task is to define a complex structure $J$ on $V_{\bR}$, where $V = \Res_{F/\bQ} \tV$.
Now, $V_{\bR} = \bigoplus_{\sigma \in \sG} \tV_{\bR}^{\sigma}$, so it is sufficient to define a complex structure $\tJ^{\sigma}$ on $\tV_{\bR}^{\sigma}$ for each $\sigma \in \sG$.
When $\sigma$ is the identity, we have already defined $\tJ$ on $\tV_{\bR}$.
In the same manner we can define $\tJ^{\sigma}$ for each $\sigma \in \sG$, such that $\tE^{\sigma} (x, \tJ^{\sigma} y)$ is symmetric and positive definite on $\tV_{\bR}^{\sigma}$.
Then $J = \sum_{\sigma \in \sG} \tJ^{\sigma}$ is a complex structure on $V_{\bR}$.

\subsection{Conclusion of the proof}
Since each $\tE^{\sigma} (x, \tJ^{\sigma} y)$ is symmetric, so is $E(x, Jy)$.
It remains to show that $E(x, Jy)$ is positive definite.
For each $\sigma \in \sG$, let $\alpha(\sigma)$ be the unique element of $M^{\sigma} \cap S_0$.
Then, we have
\begin{align*}
	E(x, Jx) &= \sum_{\sigma \in \sG} \tE^{\sigma} (x, \tJ^{\sigma}x) \\
	&= \sum_{\sigma \in \sG} \Bigg( \tE_{\alpha(\sigma)} (x_{\alpha(\sigma)}, \tJ^{\alpha(\sigma)}x_{\alpha(\sigma)})
		\prod_{\substack{\beta \in M^{\sigma} \\ \beta \neq \alpha(\sigma)}} \gamma_{\beta} (x_{\beta}, x_{\beta}) \Bigg) \\
	&= \sum_{\sigma \in \sG} Q^{\sigma} (x),
\end{align*}
where
\[
	Q^{\sigma} (x) = \tE_{\alpha(\sigma)} (x_{\alpha(\sigma)}, \tJ^{\alpha(\sigma)}x_{\alpha(\sigma)})
		\prod_{\substack{\beta \in M^{\sigma} \\ \beta \neq \alpha(\sigma)}} \gamma_{\beta} (x_{\beta}, x_{\beta})
\]
is a symmetric form on $V_{\bR}$.
For $\sigma$ equal to the identity, we know that $Q(x) = \tE(x,Jx)$ is positive definite.
Therefore there exists a positive integer $N$ such that
\[
	Q'(x) = NQ(x) + \sum_{\substack{\sigma \in \sG \\ \sigma \neq \id}} Q^{\sigma}(x)
\]
is positive definite (see Addington \cite{Addington}*{Lemma 4.9, p.~80}).
For each $j$ ($1 \leq j \leq t$), let $c_j \in F_j$ be such that $\alpha(c_j) > N$ if $\alpha \in S_0$, and $0 < \alpha(c_j) < 1$ if $\alpha \notin S_0$.
Replace each $\tE_{\alpha}$ by $c_{j(\alpha)}\tE_{\alpha}$.
Then $E(x, Jy)$ is positive definite.

An $H$-element for $Sp(V, E)$ is given by $\frac{1}{2} J$.
Since $H_0 = \sum_{\alpha \in S_0} H_{0, \alpha}$ is an $H$-element of $G$, and $\rho_{\alpha}$ satisfies the $H_1$-condition with respect to $H_{0, \alpha}$ and $\frac{1}{2} J_{\alpha}$ whenever $\alpha \in S_0$, it follows from our construction that $\rho$ satisfies the $H_1$-condition, and therefore defines a Kuga fiber variety.

\section{The general case}

We next derive the general case from the rigid case.

\begin{theorem}
\label{maintheorem}
Let $\rho$ be a representation of $G$ satisfying Satake's conditions in \S\textup{\ref{necessary}}.
Then some multiple of $\rho$ defines a Kuga fiber variety.
\end{theorem}

\begin{proof}
We keep the notations used in the proof of Theorem \ref{rigid}.
Without loss of generality we assume that $\rho$ is a primary representation, i.e., a multiple of an irreducible representation.
An irreducible subrepresentation $\rho_0$ of $\rho_{\bC}$ is said to be \emph{rigid} if it is nontrivial on some noncompact factor of $G(\bR)$.
We define the \emph{index of rigidity} of $\rho$ to be the cardinality of the set $\{ \sigma \in \sG \mid \mu^{\sigma} \text{ is rigid} \}$, where $\mu$ is an irreducible subrepresentation of $\rho_{\bC}$.
It depends only on $\rho$, and not on the choice of $\mu$.

Suppose $\rho$ is not rigid. Then there exists a subrepresentation $\mu$ of $\rho_{\bC}$ such that $\mu$ is trivial on all noncompact factors of $G_{\bR}$.
Let $\sG_1 = \{ \sigma \in \sG \mid \mu^{\sigma} = \mu \}$.
Then $\rho_{\bC}$ is equivalent to a multiple of $\sum_{\sigma \in \sG/\sG_1} \mu^{\sigma}$.

Let $\alpha_0 \in S_j$ be such that $\mu$ is nontrivial on $G_{\alpha_0}$.
Extend $\alpha_0$ to an embedding of $F$ into $\bR$, and denote it again by $\alpha_0$.
Let $B$ be a quaternion algebra over $F$ which splits at $\alpha_0$ and ramifies at all other infinite places.
Let $SL_1(B)$ be the group of norm $1$ units of $B$, and $H = \Res_{F/\bQ} SL_1(B)$.
Then $H(\bR) = \prod_{\alpha \in \overline{S}} H_{\alpha}$, where $\overline{S}$ is the set of embeddings of $F$ into $\bR$, and $H_{\alpha} = H \otimes_{F, \alpha} \bR$.

Define a representation of $G \times H$  by
\[
	\rho_1 = \sum_{\sigma \in \sG} \mu^{\sigma} \otimes p_{\alpha_0}^{\sigma},
\]
where $p_{\alpha_0} \colon H(\bR) \to H_{\alpha_0} = SL_2(\bR)$ is the projection.
Since $\rho_1$ is invariant under any automorphism of $\bC$, some multiple $n \rho_1$ of $\rho_1$ is defined over $\bQ$.
Now $p_{\alpha_0}^{\sigma} = p_{\alpha_0^{\sigma}}$.
If $\sigma$ is not the identity then $\alpha_0^{\sigma} \neq \alpha_0$.
Since $H_{\alpha}$ is compact for $\alpha \neq \alpha_0$, we see that $n \rho_1$ satisfies the stability condition.
We verify that $n \rho_1$ satisfies all of Satake's conditions.

Now, $\mu^{\sigma} \otimes p_{\alpha}^{\sigma}$ is rigid whenever $\mu^{\sigma}$ is rigid, and it is also rigid when $\sigma$ is the identity.
Hence the index of rigidity of $n \rho_1$ is greater than the index of rigidity of $\rho$.
Continuing this process, if necessary, we will eventually get a representation $\tilde{\rho}$ whose index of rigidity is the cardinality of $\sG$, i.e., one which is rigid.
Then Theorem~\ref{rigid} implies that some multiple of $\tilde{\rho}$ defines a Kuga fiber variety.
Since the restriction of $\tilde{\rho}$ to $G$ is a multiple of $\rho$, this completes the proof.
\end{proof}

\section{An example}
\label{example}

Let $F = \bQ (\sqrt{3})$.
Let $\alpha_1, \alpha_2$ be the embeddings of $F$ into $\bR$.
Let $B$ be a quaternion algebra over $F$ which splits at $\alpha_1$ and ramifies at $\alpha_2$.
Let $\tG_1$ be the group of norm $1$ units in $B$, and, $G_1 = \Res_{F/\bQ} \tG_1$.
Let $\tG_2$ be a group over $F$ such that $\tG_2 \otimes_{F, \alpha_1} \bR = SU(5,1)$ and
$\tG_2 \otimes_{F, \alpha_2} \bR = SU(6,0)$.
Let $G_2 = \Res_{F/\bQ} \tG_2$.
Let $G = G_1 \times G_2$. Then
\[
	G(\bR) = SL_2(\bR) \times SU(2) \times SU(5, 1) \times SU(6, 0).
\]

We shall classify all representations of $G$ which define Kuga fiber varieties.
Let
\begin{align*}
	& p_1 \colon G(\bR) \to SL_2(\bR),\\
	& p_2 \colon G(\bR) \to SU(2),\\
	& p_3 \colon G(\bR) \to SU(5, 1),\\
	& p_4 \colon G(\bR) \to SU(6, 0),
\end{align*}
be the projections.
Then, the representations of $G$ defining Kuga fiber varieties are equivalent over $\bR$ to linear combinations of the following:
\begin{enumerate}
\item $p_1 \oplus p_2$,
\item $p_1 \otimes p_2$,
\item $\bigwedge^k p_3 \oplus \bigwedge^k p_4$,
\item $\left( \bigwedge^j p_3 \otimes \bigwedge^k p_4 \right) \oplus \left( \bigwedge^k p_3 \otimes \bigwedge^j p_4 \right)$,
\item $\left( p_1 \otimes \bigwedge^k p_4 \right) \oplus \left( p_2 \otimes \bigwedge^k p_3 \right)$.
\end{enumerate}

Of these, the first four are direct sums of representations of either $G_1$ alone, or $G_2$ alone.
The last one is a representation of the product group in an essential manner.
A general fiber of the corresponding Kuga fiber variety is a simple abelian variety; the semisimple part of its Hodge group is isogenous to $G = G_1 \times G_2$.

\end{document}